\documentclass[letterpaper, 10 pt, conference]{ieeeconf}  

\IEEEoverridecommandlockouts                              

\usepackage{amsmath,amssymb,amsfonts}
\usepackage{algorithmic}
\usepackage{graphicx}
\usepackage{textcomp}
\usepackage{graphicx}

\usepackage{amsthm}
\usepackage{enumerate}
\usepackage{soul,xcolor}
\usepackage{amstext}
\usepackage{tikz}
\usepackage{algorithmic}
\usepackage{algorithm}
\usepackage{siunitx}
\usepackage{balance}
\usepackage{cuted}
\usepackage{array}
\usepackage{float}
\usepackage{multirow}
\usepackage{psfrag}
\usepackage{accents}
\usepackage{etoolbox}
\usepackage{url}
\usepackage[normalem]{ulem}

\theoremstyle{plain}
\newtheorem{thm}{Theorem}
\newtheorem{lem}{Lemma}
\newtheorem{prop}{Proposition}

\theoremstyle{definition}
\newtheorem{defn}{Definition}
\newtheorem{assum}{Assumption}

\theoremstyle{remark}
\newtheorem{rem}{Remark}

\DeclareMathOperator*{\argmax}{argmax}
\DeclareMathOperator*{\argmin}{argmin}

\newcommand{\R}{\mathbb{R}}

\newcommand{\one}{\mathbf{1}}
\newcommand{\onen}{\one_n\one_n^T/n}
\newcommand{\diag}{\mathrm{diag}}
\newcommand{\Dout}{D^{\mathrm{out}}}
\newcommand{\ip}[2]{\langle {#1}, {#2} \rangle}

\DeclareMathOperator{\tr}{Tr}
\newcolumntype{L}[1]{>{\centering\let\newline\\\arraybackslash\hspace{0pt}}m{#1}}

\let\classAND\AND
\let\AND\relax
\usepackage{algorithmic}

\let\AND\classAND
\AtBeginEnvironment{algorithmic}{\let\AND\algoAND}

\floatname{algorithm}{Algorithm}

\title{\LARGE \bf
Distributed ADMM with linear updates over directed networks*
}

\author{Kiran Rokade$^{1}$ and Rachel Kalpana Kalaimani$^{2}$
\thanks{*This work has been partially supported by DST-INSPIRE Faculty Grant, Department of Science and Technology (DST), Govt. of India (ELE/16-17/333/DSTX/RACH).}
\thanks{$^{1}$Kiran Rokade is with the Department of Electrical and Computer Engineering, Cornell University, Ithaca, NY 14850 USA. 
        {\tt\small kvr36@cornell.edu}}%
\thanks{$^{2}$Rachel Kalpana Kalaimani is with the Department of Electrical Engineering, Indian Institute of Technology Madras, Chennai 600036 India. 
        {\tt\small rachel@ee.iitm.ac.in}}%
}

\begin{document}

\setstcolor{red}

\maketitle
\thispagestyle{empty}
\pagestyle{empty}

\begin{abstract}

We propose a distributed version of the Alternating Direction Method of Multipliers (ADMM) with linear updates for directed networks. We show that if the objective function of the minimization problem is smooth and strongly convex, our distributed ADMM algorithm achieves a geometric rate of convergence to the optimal point. Our algorithm exploits the robustness inherent to ADMM by not enforcing accurate consensus, thereby significantly improving the convergence rate. We illustrate this by numerical examples, where we compare the performance of our algorithm with that of state-of-the-art ADMM methods over directed graphs. 
\end{abstract}


\section{Introduction}

In this paper, we focus on solving an unconstrained optimization problem of the form
\begin{equation}
    \label{eq:opt_prob_intro} \tag{P1}
    \begin{aligned}
        & \min_{x \in \R^m} & & \sum_{i = 1}^n f_i(x),
    \end{aligned}    
\end{equation}
where $x$ is the decision variable. The computations involved in solving the problem are performed by a group of $n$ agents which form a communication network. Due to the restrictions imposed by the network, an agent can communicate only with its neighbours in the network. For each $i$, $f_i: \R^{m} \rightarrow \R$ is a convex function privately known to agent $i$. The goal of each agent is to solve \eqref{eq:opt_prob_intro} in cooperation with its neighbours, without explicitly revealing its function $f_i$. This restriction on the exchange of $f_i$'s could arise due to privacy reasons or due to the communication cost involved in such an exchange. Since each agent minimizes the sum of all functions, it computes a ``globally optimal'' solution without revealing its own function to other agents. 
Such distributed optimization problems are encountered in several applications such as achieving average consensus in a distributed manner, distributed power system control, formation control of robots, statistical learning etc \cite{nedich_survey_dist_opt}. 

Distributed optimization problems have been widely studied in the literature. Existing methods for solving these problems can be broadly classified into two types: gradient-descent-based primal methods \cite{nedich_subgradient_push,nedich_dig,makhdoumi_bal_weights} and Lagrangian-based dual-ascent methods \cite{wei_async_admm,ozdaglar_admm,panda,dcdistadmm}. While most of the methods in the literature assume that the communication network is undirected \cite{wei_async_admm,ozdaglar_admm,panda}, i.e., communication links between agents are bidirectional; in practice, the network is often directed, i.e., the links are unidirectional. This can occur, for example, if two agents have different broadcast ranges. Generalizing distributed optimization algorithms, which work for undirected graphs, to directed graphs is not straightforward, and in the past has required the introduction of some novel concepts. For example, gradient-descent methods were generalized to directed networks by introducing new concepts such as balancing weights \cite{makhdoumi_bal_weights}, push-sum \cite{nedich_subgradient_push}, push-DIGing (Distributed In-exact Gradient tracking) \cite{nedich_dig} etc. On the other hand, there has been less progress in generalizing Lagrangian-based methods to directed graphs. 

A special type of Lagrangian-based dual-ascent method is the Alternating Direction Method of Multipliers (ADMM). Compared to the standard dual-ascent method, ADMM is more robust and amenable to parallel implementation in a multi-agent setting \cite{boyd_admm}. 
While, distributed implementation of ADMM is well-studied for \emph{undirected} graphs \cite{wei_async_admm,ozdaglar_admm}, an extension to \emph{directed} graphs was non-existent in the literature (see \cite{nedich_dig}), until recently. 
In the past few years, to the best of our knowledge, two methods for distributed ADMM over directed graphs have been proposed: DC-DistADMM (Directed Constrained Distributed ADMM) in \cite{dcdistadmm}, and D-ADMM-FTERC (Distributed Alternating Direction Method of Multipliers using Finite-Time Exact Ratio Consensus) in \cite{finite-time_admm}.

The key differences between our method, and DC-DistADMM \cite{dcdistadmm} and D-ADMM-FTERC \cite{finite-time_admm} are as follows.
\begin{itemize}
    \item At each iteration of the respective algorithms, DC-DistADMM requires that, for a fixed $\epsilon > 0$, the agents reach consensus up to an $\epsilon$-accuracy, while D-ADMM-FTERC requires that the agents reach \emph{exact} consensus. 
    As noted earlier, ADMM is known to be robust to errors in implementation \cite{boyd_admm}. For this reason, we do not enforce any accuracy on the consensus step. Instead, we require the agents to communicate with each other $B \geq 1$ times per iteration. The value of $B$ is chosen depending on other parameters of the algorithm. However, it is observed in simulations that even with $B = 1$, the algorithm performs well in the sense that it gives a convergence rate that is better than that of DC-DistADMM and D-ADMM-FTERC. 
    \item Instead of the push-sum and finite-time ratio-consensus methods, we use the method of balancing weights to achieve consensus among the agents. The balancing weights are updated dynamically in a distributed manner. While the operations involved in push-sum and ratio-consensus methods involve divisions, which may cause numerical issues, the balancing weights method involves only addition and multiplication (linear updates), and hence does not suffer from these problems.
    \item Inspired by the proofs in \cite{nedich_dig,panda}, we prove convergence of our method using the \emph{small gain theorem}. This technique is quite different to the ones used in \cite{dcdistadmm} and \cite{finite-time_admm}, and may be of interest independently.
\end{itemize}

Our main contributions to the literature are as follows. 
\begin{itemize}
    \item We propose an ADMM algorithm (Algorithm \ref{alg:main}) which solves the optimization problem \eqref{eq:opt_prob_intro} in a distributed manner over directed networks. At each iteration, our algorithm runs an ``inner loop'' where each agent computes an \emph{approximate} average of the primal variable iterates of all agents in a distributed manner.
    This approximate consensus step exploits the inherent robustness of ADMM, resulting in a superior performance compared to other distributed ADMM algorithms in \cite{dcdistadmm}, \cite{finite-time_admm}.
    \item Under the assumption that each $f_i$ in \eqref{eq:opt_prob_intro} is strongly convex and smooth, we show that the primal-dual iterates of the algorithm converge to their unique optimal points at a geometric rate, provided an explicit bound on the penalty parameter of the ADMM algorithm is satisfied (Theorem \ref{thm:main}). 
    \item Through numerical examples, we demonstrate the superior performance of our algorithm compared to state-of-the-art ADMM methods over directed graphs, in particular \cite{dcdistadmm} and \cite{finite-time_admm}.
    We also demonstrate that our algorithm is robust to changes in its parameters.
\end{itemize}



The rest of the paper is organized as follows. In Section \ref{sec:prob_form}, we propose a problem equivalent to \eqref{eq:opt_prob_intro} that is amenable to distributed implementation. In Section \ref{sec:algorithm}, we present our distributed ADMM algorithm for solving this equivalent problem. In Section \ref{sec:main_res}, we state our main result which guarantees convergence of the algorithm. We also give an outline of the proof of this result. In Section \ref{sec:num_examples}, we present some numerical examples to compare the performance of our algorithm with some other ADMM algorithms over directed graphs. Finally, we give some concluding remarks in Section \ref{sec:conclusion}. Proofs of all the results can be found in the appendix.

\noindent \textbf{Notation:} Let $\log$ denote logarithm with base $10$. $\one_n$ be the vector of all ones in $\R^n$. For a vector $x \in \R^n$, $x_i$ be the $i$th element of $x$, $x^T$ be the transpose of $x$ and let $\|x\|$ denote the $2$-norm of $x$. 

For a matrix $X \in \R^{n \times m}$, let $x_i \in \R^m$ denote its $i$th row. Let $\|X\|$ be its Frobenius norm and $\|X\|_2$ be its induced $2$-norm. Given two matrices $X$ and $Y$, let $\ip{X}{Y} := \tr(X^T Y)$ be their Frobenius inner-product. Given a vector $x \in \R^n$, let $\diag(x) \in \R^{n \times n}$ be the diagonal matrix with elements $x_1, \dots, x_n$ on the diagonal. 

For a function $f : \R^{n \times m} \rightarrow \R$, let $f^* : \R^{n \times m} \rightarrow \R$ denote the convex conjugate of $f$ defined as $f^*(A) = \sup_{X \in \R^{n \times m}} \ip{A}{X} - f(X)$ for all $A \in \R^{n \times m}$ such that the supremum is finite. Let $\nabla f(X) \in \R^{n \times m}$ be the gradient of $f$ at $X$ defined as $[\nabla f(X)]_{ij} = \partial f(X)/\partial X_{ij}$. 

We use the shorthand $\{S^k\}$ to denote the sequence $\{S^k\}_{k \geq 0}$. 

\section{Problem Formulation}
\label{sec:prob_form}

Consider a set of agents denoted by $V = \{1, \dots, n\}$. The communication pattern between the agents is depicted by a directed graph $G = (V,E)$, where $E$ is the set of all directed edges. We denote $(i,j) \in E$ if there exists a directed edge from agent $j$ to agent $i$. Following assumption is standard in the literature when agents desire to reach consensus over a directed graph, e.g., \cite{makhdoumi_bal_weights,panda,nedich_subgradient_push}.
\begin{assum}
    \label{assum:G_strongly_conn}
    $G$ is strongly connected, i.e., there exists a directed path between each pair of agents.
\end{assum}

The agents must cooperatively solve the optimization problem \eqref{eq:opt_prob_intro}, which can be equivalently written as
\begin{equation}
    \label{eq:opt_prob_main} \tag{P2}
    \begin{aligned}
        & \mathrm{min} & & \sum_{i = 1}^n f_i(x_i) \\
        & \mathrm{s.t.} & & x_i = x_j \textrm{ for all } i,j \in V,
    \end{aligned}    
\end{equation}
where $f_i: \R^{m} \rightarrow \R$ is the convex function known only to agent $i$, and $x_i \in \R^{m}$ is the decision variable of agent $i$. The constraint $x_i = x_j \textrm{ for all } i,j \in V$ is called the \emph{consensus constraint}. 
 To write \eqref{eq:opt_prob_main} in a compact form, let $X = \begin{bmatrix} x_1 \ \dots \ x_n \end{bmatrix}^T \in \R^{n \times m}$ be the matrix of all decision variables. Then, the consensus constraint can be written as 
\begin{equation*}
    X = (\onen) X.
\end{equation*}
Further, let $f(X) = \sum_{i = 1}^n f_i(x_i)$. Then, \eqref{eq:opt_prob_main} can be written as
\begin{equation}
    \label{eq:opt_prob_vector_form} \tag{P3}
    \begin{aligned}
        & \mathrm{min}_{X \in \R^{n \times m}} & & f(X) \\
        & \mathrm{s.t.} & & X = (\onen) X.
    \end{aligned}    
\end{equation}
We assume that \eqref{eq:opt_prob_vector_form} is solvable, i.e., there exists an optimal point $X^*$ of \eqref{eq:opt_prob_vector_form}. We make the following assumptions on the function $f$. These assumptions are standard in the literature whenever a geometric rate of convergence is desired, e.g., \cite{nedich_dig,panda,ozdaglar_admm}.
\begin{assum}
    \label{assum:strong_conv_lip_grad}
    For each $i \in \{1,\dots,n\}$ the function $f_i$ is $\mu$-strongly convex, i.e., $\exists \mu > 0$ such that for all $x,y \in \R^m$,
    \begin{equation*}
        f_i(y) \geq f_i(x) + \nabla f_i(x)^T(y_i - x_i) + \frac{\mu}{2}\|y_i - x_i\|^2.
    \end{equation*}
    Further, for each $i \in \{1,\dots,n\}$, $f_i$ is differentiable and $\nabla f_i$ is $L$-Lipschitz, i.e., $\exists L > 0$ such that for all $x,y \in \R^m$,
    \begin{equation*}
        \|\nabla f_i(x) - \nabla f_i(y)\| \leq L \|x_i - y_i\|.
    \end{equation*}
\end{assum}

An immediate consequence of Assumption \ref{assum:strong_conv_lip_grad} is that $f$ is $\mu$-strongly convex and $\nabla f$ is $L$-Lipschitz.

The next two steps ensure that our ADMM algorithm to solve \eqref{eq:opt_prob_vector_form} can be executed in a distributed and parallel manner over the directed graph $G$. First, we introduce a new variable $Y \in \R^{n \times m}$ and write \eqref{eq:opt_prob_vector_form} as
\begin{equation}
    \label{eq:opt_prob_y} \tag{P4}
    \begin{aligned}
        & \mathrm{min}_{X, Y \in \R^{n \times m}} & & f(X) \\
        & \mathrm{s.t.} & & X = Y, \ Y = (\onen) Y.
    \end{aligned}    
\end{equation}
This decouples the decision variables $x_1, \dots, x_n$ of the agents from each other. Second, we move the constraint $Y = (\onen) Y$ into the objective function using an indicator function as follows. Define
\begin{equation*}
    \label{eq:indicator_fn}
    I(Y) = \begin{cases} 0 & \textrm{ if } Y = (\onen) Y, \\ 
    \infty & \textrm{ otherwise}. \end{cases}
\end{equation*}
Now, \eqref{eq:opt_prob_y} is equivalent to
\begin{equation}
    \label{eq:opt_prob_indicator_fn} \tag{P5}
    \begin{aligned}
        & \mathrm{min}_{X, Y \in \R^{n \times m}} & & f(X) + I(Y) \\
        & \mathrm{s.t.} & & X = Y.
    \end{aligned}    
\end{equation}
The formulation above enables us to impose the constraint $Y = (\onen) Y$ explicitly at each step of the ADMM algorithm, as we shall see later. Since $f$ is strongly convex under Assumption \ref{assum:strong_conv_lip_grad}, \eqref{eq:opt_prob_indicator_fn} has a unique optimal point $X^* = Y^*$ which satisfies $(\onen)X^* = Y^*$.

To solve \eqref{eq:opt_prob_indicator_fn} using ADMM, we define the augmented Lagrangian of the problem as
\begin{equation}
    \label{eq:aug_lagrangian}
    L_\rho(X,Y,A) = f(X) + I(Y) + \ip{A}{X-Y} + \frac{\rho}{2}\|X - Y\|^2, 
\end{equation}
where $A \in \R^{n \times m}$ is the dual variable associated with the constraint $X = Y$ and $\rho > 0$ is the \emph{penalty parameter}. The term $(\rho/2)\|X - Y\|^2$ is called the \emph{penalty term}. Let $A^*$ be a dual optimal point of \eqref{eq:opt_prob_indicator_fn}. In the next section, we propose our distributed ADMM algorithm which generates sequences $\{A^k\}$ and $\{X^k\}$ of primal-dual iterates which converge to $A^*$ and $X^*$ respectively, at a geometric rate.

\section{Algorithm}
\label{sec:algorithm}

We begin by writing the standard 2-block ADMM algorithm (see  \cite{boyd_admm}) for \eqref{eq:opt_prob_indicator_fn} as follows. For simplicity, all iterates are initialized at zero. For each $k \geq 0$, 
\begin{align}
    \label{eq:x_update_original}
    X^{k+1} &= \argmin_{X \in \R^{n \times m}} L_\rho(X,Y^{k},A^k), \\
    \label{eq:y_update_original}
    Y^{k+1} &= \argmin_{Y \in \R^{n \times m}} L_\rho(X^{k+1},Y,A^k), \\
    \label{eq:a_update}
    A^{k+1} &= A^k + \rho(X^{k+1} - Y^{k+1}).
\end{align}

We assume that for all $k \geq 0$, each agent $i$ maintains the set of iterates $\{x_i^k,y_i^k,a_i^k\}$. To implement the algorithm above in a distributed manner, each agent must update its iterates using only the information received from its in-neighbours. We analyze each update step above to see if such a distributed implementation is possible. \\
\textbf{The $X$ update step:} Substituting for $L_\rho$ from \eqref{eq:aug_lagrangian} in 
\eqref{eq:x_update_original}, we obtain
\begin{equation}
\label{eq:x_update_explicit}
    X^{k+1} = \argmin_{X \in \R^{n \times m}} f(X) + \ip{A^k}{X} + \frac{\rho}{2}\|X - Y^k\|^2.
\end{equation}
Note that, by definition, $f$ can be decomposed as $f(X) = \sum_{i = 1}^n f_i(x_i)$. Hence, the update step above is equivalent to
\begin{equation}
\label{eq:x_i_update}
    x_i^{k+1} = \argmin_{x_i \in \R^m} f_i(x_i) + (a_i^k)^T x_i + \frac{\rho}{2}\|x_i - y_i^{k}\|^2 
\end{equation}
for all $i \in V$, which can be implemented parallely. 

\textbf{The $Y$ update step:} We can derive an explicit expression for $Y^{k+1}$ in \eqref{eq:y_update_original} as shown by the following result.
\begin{lem}
    \label{lem:y_update}
    The $Y$ update step given in \eqref{eq:y_update_original} is equivalent to
    \begin{equation}
        \label{eq:y_update_exact}
        y_i^{k+1} = \frac{1}{n} \sum_{j = 1}^n \Bigg(x_j^{k+1} + \frac{a_j^k}{\rho}\Bigg) \quad \text{ for all } i \in V.
    \end{equation}
\end{lem}
The proof of Lemma \ref{lem:y_update} can be found in Appendix \ref{sec:appendix_proof_of_y_update_lemma}. 
Note that \eqref{eq:y_update_exact} involves computing an average for each iteration. We use the idea of balancing weights \cite{makhdoumi_bal_weights}, along with a dynamic average consensus type update rule \cite{dynamic_avg_consensus} to compute an estimate of $y_i^{k+1}$ in a distributed manner. This has been proposed earlier for primal gradient tracking, e.g., \cite{makhdoumi_bal_weights,nedich_dig}. We use these techniques to track the average of the dual variable. 



Our update rule to estimate $y_i^{k+1}$ is as follows. Consider an agent $i \in V$ and an iteration $k \geq 0$. Let the agent's estimate of $y_i^{k+1}$ be denoted by $\zeta_i^{k+1}(\cdot)$. This estimate is updated dynamically. Let the estimate be initialized as $\zeta_i^{k+1}(0) = y_i^k + x_i^{k+1} - x_i^k + (a_i^k - a_i^{k-1})/\rho$.\footnote{We use the notation $a^{-1} = 0$.} The intuition behind this initialization is to perturb $y_i^k$, the last-known estimate of the average, by $x_i^{k+1} - x_i^k + (a_i^k - a_i^{k-1})/\rho$, the change in agent $i$'s contribution to the average. Such an initialization is typical of a dynamic average consensus type update rule \cite{tutorial_dac}. Note that this initialization implicitly assumes that agent $i$ knows the exact value of $y_i^k$, which may not be possible. This assumption can be relaxed with the following observation. From \eqref{eq:a_update}, we know that $x_i^k - y_i^k = (a_i^k - a_i^{k-1})/\rho$. Hence, $\zeta_i^{k+1}(0) = x_i^{k+1}$. Thus, initializing the estimate $\zeta_i^{k+1}(\cdot)$ only requires the knowledge of $x_i^{k+1}$, which the agent has from the update performed in \eqref{eq:x_i_update}. Now, agent $i$ updates its estimate as follows. Consider an integer $B \geq 1$, which must satisfy a lower-bound to be specified later. In each iteration $k \geq 0$, the agents run an ``inner loop'' $B$ number of times where they communicate and update their estimate $\zeta_i^{k+1}(\cdot)$. Specifically, for all $k \geq 0$ and $b \in \{0, \dots, B-1\}$,
\begin{align}
\label{eq:zeta_i_update}
    \zeta_i^{k+1}(b+1) &= \zeta_i^{k+1}(b)\left(1 - d_i^\mathrm{out}w_i^{k+1}(b)\right) \nonumber \\
    & \quad + \sum_{j \in V: (i,j) \in E} w_j^{k+1}(b) \zeta_j^{k+1}(b), 
\end{align}
where $d_i^\mathrm{out}$ is the out-degree of agent $i$, and $w_j^k(b) \in \R$ is the \emph{weight} used by agent $j$ to scale its outgoing information. Thus, agent $i$'s estimate is computed iteratively by taking a weighted average of its own estimate with those of its in-neighbours. To ensure that the agents reach consensus when the underlying graph is directed, the weights must be chosen appropriately. One such set of weights are the ``balancing weights'' of a graph. We next 
describe our update rule used to compute the weights in \eqref{eq:zeta_i_update}, as introduced in \cite{makhdoumi_bal_weights}.

Let $\Dout := \diag(d_1^\mathrm{out}, \dots, d_n^\mathrm{out})$ be the diagonal matrix of out-degrees of the agents. Let $C \in \R^{n \times n}$ be the adjacency matrix of the graph, i.e., $C_{ij} = 1$ if $(i,j) \in E$ and $C_{ij} = 0$ otherwise. Given a vector $w \in \R^n$ of weights, let $W = I - (\Dout - C)\diag(w)$ be the \emph{weight matrix} associated with $w$. 

\begin{defn}{(\emph{Balancing weights})}
Given a graph $G$, a vector $w$ is said to be a vector of balancing weights for $G$ if the weight matrix $W = I - (\Dout - C)\diag(w)$ associated with $w$ is doubly-stochastic, i.e., $W\one = \one, \one^T W = \one^T$.
\end{defn}


We use the same update rule as described in \cite{makhdoumi_bal_weights} to compute the set of weights required in \eqref{eq:zeta_i_update}.
Let $d_*^\mathrm{out} := \max_{i \in V} d_i^\mathrm{out}$ be the maximum out-degree of the graph and $D$ be the diameter of the graph. Following \cite{makhdoumi_bal_weights}, we initialize the weights as $w_i^1(0) \leq (1/d_*^\mathrm{out})^{2D+1}$. Then, for all $k \geq 0$ and $b \in \{0, \dots, B-1\}$,
\begin{align}
\label{eq:w_i_update}
    w_i^{k+1}(b+1) = \frac{1}{2} \Bigg(w_i^{k+1}(b) + \frac{1}{d_i^\mathrm{out}} \sum_{j \in V: (i,j) \in E} w_j^{k+1}(b)\Bigg)
\end{align}
and $w_i^{k+2}(0) = w_i^{k+1}(B)$. Thus, agent $i$'s weight is updated by taking a weighted average of its own weight with those of its in-neighbours. Hence, \eqref{eq:w_i_update} can be implemented in a distributed manner. Note that for each iterate $b \in \{0, \dots, B-1\}$, the update rule above is implemented sequentially after implementing \eqref{eq:zeta_i_update}.

The update rule in \eqref{eq:w_i_update} can be written compactly as
\begin{align}
\label{eq:w_update}
    w^{k+1}(b+1) = P w^{k+1}(b)
\end{align}
and $w^{k+2}(0) = w^{k+1}(B)$, where $P := (I + (\Dout)^{-1}C)/2$. Further, the update rule in \eqref{eq:zeta_i_update} can be written compactly as
\begin{align}
\label{eq:zeta_update}
    \zeta^{k+1}(b+1) = W^{k+1}(b) \zeta^{k+1}(b)
\end{align}
where $\zeta^{k+1}(B) = \begin{bmatrix} \zeta^k_1(B) \ \dots \ \zeta^k_n(B) \end{bmatrix}^T$ and $W^{k+1}(b) := I - (\Dout - C)\diag(w^{k+1}(b))$ is the weight matrix associated with the weights $w^{k+1}(b)$. It is known that the weights updated according to \eqref{eq:w_update} converge to the set of balancing weights of the graph \cite{makhdoumi_bal_weights}. 
Additionally, we require that the weight matrix $W^k(b)$ satisfies certain properties. These properties are stated in our next result.

\begin{lem}
\label{lem:W_properties}
Let $W^k(b) = I - (\Dout - C)\diag(w^k(b))$ be the weight matrix associated with the weights $w^k(b)$ updated as per \eqref{eq:w_update}. The matrix has the following properties.
\begin{enumerate}
    \item For all $k \geq 1$, for all $b \in \{0, \dots, B-1\}$, $\exists p^{k}(b) \in \R^n$ such that $p^{k}(b)^T \one = 1$, $\one^T W^{k}(b) = \one^T$ and $W^{k}(b) p^{k}(b) = p^{k}(b)$.  
    \item $\exists \bar{k} \geq 0, \exists \delta \in [0,1)$ such that $\|W^k(b) - p^k(b) \one^T\| \leq \delta$ for all $k \geq \bar{k}$, for all $b \in \{0, \dots, B-1\}$. 
    \item For all $b \in \{0, \dots, B-1\}$, $\lim_{k \rightarrow \infty} p^k(b) = \one/n$.
    \item $\exists \hat{k} \geq 0, \exists M \geq 0$ such that $\|W^k(b)\| \leq M$ for all $k \geq \hat{k}$, for all $b \in \{0, \dots, B-1\}$. 
\end{enumerate}
\end{lem}

The proof of Lemma \ref{lem:W_properties} can be found in Appendix \ref{sec:appendix_proof_of_W_lemma}. In essence, Lemma \ref{lem:W_properties} states that $W^k(b)$ is left-stochastic, while its right eigenvector associated with the eigenvalue $1$ converges to $\one$ as $k \rightarrow \infty$. Further, for large $k$, $W^k(b)$ has a bounded norm and its second-largest eigenvalue is strictly inside the unit circle.

Finally, we look at the distributed implementation of the dual update step.

\textbf{The $A$ update step:} From \eqref{eq:a_update}, we have $a_i^{k+1} = a_i^k + \rho(x_i^{k+1} - y_i^{k+1})$. Thus, each agent $i \in V$ can compute $a_i^{k+1}$ independently using its set of iterates $\{x_i^{k+1},y_i^{k+1}\}$. 

\begin{rem}
    The agents need to communicate with their neighbours only during the $Y$ update step. 
\end{rem}

The steps implemented by each agent $i \in V$ are summarized in Algorithm \ref{alg:main}. 
\begin{algorithm}
\caption{Distributed ADMM at agent $i \in V$}
\label{alg:main}
\begin{algorithmic}
    \STATE Initialize $x_i^0 = y_i^0 = a_i^0 = 0, w_i^1(0) \leq (1/d_*^\mathrm{out})^{2D+1}$.
    \FOR{$k = 0, 1, \dots$}
	\item Compute $x_i^{k+1}$ using \eqref{eq:x_i_update}.
    \item Initialize $\zeta_i^{k+1}(0) = x_i^{k+1}$
    \FOR{$b = 0, \dots, B-1$}
    \item Send $w_i^{k+1}(b)$ and $\zeta_i^{k+1}(b)$ to out-neighbours.
    \item Compute $\zeta_i^{k+1}(b+1)$ using \eqref{eq:zeta_i_update}.
    \item Compute $w_i^{k+1}(b+1)$ using \eqref{eq:w_i_update}.
    \ENDFOR
    \item Fix $y_i^{k+1} = \zeta_i^{k+1}(B)$.
    \item Fix $w_i^{k+2}(0) = w_i^{k+1}(B)$.
    \item Compute $a_i^{k+1} = a_i^k + \rho(x_i^{k+1} - y_i^{k+1})$.
    \ENDFOR
\end{algorithmic}
\end{algorithm}

\begin{rem}
    In Algorithm \ref{alg:main}, for the ease of notation, we relabel the final estimate $\zeta_i^{k+1}(B)$ of $y_i^{k+1}$ to $y_i^{k+1}$ itself. Thus, henceforth, $Y^k$ refers to the \emph{estimate} $\zeta^k(B)$ of $Y^k$. This is not to be confused with the \emph{exact} value of $Y^k$ defined by \eqref{eq:y_update_exact}, which will not be referred to in the rest of the paper.
\end{rem}


\section{Main Result}
\label{sec:main_res}

In this section, we state our main result and then give an outline of its proof. The details of the proof can be found in the appendix. 

\begin{thm}
    \label{thm:main}
    Given a problem of the form \eqref{eq:opt_prob_indicator_fn}, suppose that the parameters of Algorithm \ref{alg:main} are chosen such that the number of communication rounds per iteration, $B$, satisfies 
        \begin{align}
        \label{eq:B_condition}
        B \geq \max\bigg\{1,\bigg\lceil \frac{\log (\gamma_1\gamma_3)}{\log (1/\delta)} \bigg\rceil\bigg\},
    \end{align}
    the convergence rate $\lambda \in (0,1)$ satisfies 
    \begin{align}
            \label{eq:lambda_condition_1}
        \frac{L}{\mu}\left(\frac{1}{\lambda^2} -1\right) < \left(\frac{\lambda}{1+\lambda}\right) \min\left\{\frac{1}{2},\hat{c}_3\right\}
    \end{align}
    and the penalty parameter $\rho$ satisfies
    \begin{align}
        \label{eq:rho_condition_1}
        \rho \in \left(L\left(\frac{1}{\lambda^2} -1\right), \left(\frac{\lambda\mu}{1+\lambda}\right) \min\left\{\frac{1}{2},\hat{c}_3\right\}\right),
    \end{align}
    where
    \begin{align*}
        \gamma_1 &:= 1/\mu, \\
        \gamma_3 &:= \hat{c}_1 + \rho \left(1 + \frac{2\hat{c}_1}{\mu}\right) \left(\frac{\mu(1+\lambda)}{\lambda \mu - 2 \rho (1 + \lambda)}\right), \\
        \hat{c}_1 &:= \frac{1}{\hat{c}_2}\Bigg(\sqrt{\rho(L+\beta)} + \\
        & \quad \sqrt{\rho^3\Big(\frac{1}{\beta} + \frac{1}{\mu}\Big) } \left(\frac{\mu(1 + \lambda)}{\lambda \mu - 2 \rho (1 + \lambda)}\right) \Bigg), \\
        \hat{c}_2 &:= \sqrt{1 + \frac{\rho}{L} - \frac{1}{\lambda^2}} - \sqrt{\rho^3\Big(\frac{1}{\beta} + \frac{1}{\mu}\Big)} \left(\frac{2(1 + \lambda)}{\lambda \mu - 2 \rho (1 + \lambda)}\right), \\
        \hat{c}_3 &:= \frac{\lambda\mu}{4\left(\lambda\mu + L(1+\lambda)\left(1/\beta + 1/\mu\right)\right)}, \\
    \end{align*}
    $\beta > 0$ is arbitrary, $\delta \in [0,1)$, $M \geq 0$ are as defined in Lemma \ref{lem:W_properties}, and $L$ and $\mu$ are as defined in Assumption \ref{assum:strong_conv_lip_grad}. Then, the iterates generated by Algorithm \ref{alg:main} satisfy
    \begin{align*}
        \|X^k - X^*\| \leq c_x \lambda^k, \quad \|A^k - A^*\| \leq c_a \lambda^k
    \end{align*}
    for some non-negative constants $c_a, c_x$, where $(X^*,A^*)$ is the unique primal-dual optimal point of \eqref{eq:opt_prob_indicator_fn}.
\end{thm}

\begin{rem}
    \label{rem:lambda_existence}
    It is easy to check that a (large-enough) $\lambda \in (0,1)$ which satisfies \eqref{eq:lambda_condition_1} always exists. As $\lambda \rightarrow 1$, the LHS of \eqref{eq:lambda_condition_1} converges to zero, while the RHS converges to a positive number. This implies that a $\rho > 0$ satisfying \eqref{eq:rho_condition_1} always exists.
\end{rem}
\begin{rem}
    If $f$ is $\mu$-strongly convex and $L$-smooth (as is the case due to Assumption \ref{assum:strong_conv_lip_grad}), then $\kappa = L/\mu \geq 1$ is called the \emph{condition number} of $f$.
    We can argue that the convergence of Algorithm \ref{alg:main} is faster if the problem is well-conditioned, i.e., $\kappa$ is small. To see this, note that  with a decrease in $\kappa = L/\mu$, the LHS of \eqref{eq:lambda_condition_1} decreases, while the RHS increases (since $\hat{c}_3$ increases). Thus, smaller values of $\kappa \geq 1$ give a wider range of values of the convergence rate $\lambda$ that satisfy \eqref{eq:lambda_condition_1}.
\end{rem}

The proof of Theorem \ref{thm:main} uses the small gain theorem. This proof is inspired by the ideas in \cite{nedich_dig} and \cite{panda}.  We briefly recall the theorem after introducing some notation. Given a sequence $\{S^k\}$ of matrices in $\R^{n \times m}$ and a convergence rate $\lambda \in (0,1)$, let
\begin{align}
    \label{eq:defn_lambda_norm}
    \|S\|^{\lambda,K} = \max_{k = 0, \dots, K} \frac{\|S^k\|}{\lambda^k}, \quad \|S\|^{\lambda} = \sup_{k \geq 0} \frac{\|S^k\|}{\lambda^k}
\end{align}
for all $K \geq 0$. Further, given two sequences $\{S_1^k\},\{S_2^k\}$ and a constant $\gamma \geq 0$, we use the notation $S_1 \xrightarrow{\gamma} S_2$ to denote $
    \|S_2\|^{\lambda,K} \leq \gamma \|S_1\|^{\lambda,K} + \omega$
for all $K \geq 0$, for some constant $\omega \geq 0$. 
Now, the small gain theorem can be stated as follows.
\begin{prop}{\cite[Theorem 3.7]{nedich_dig} (Small gain theorem)}
    \label{prop:small_gain_thm}
    Given $m$ sequences $\{S_1^k\}, \dots, \{S_m^k\}$ of vectors in $\R^n$, suppose there exist non-negative constants $\gamma_1, \dots, \gamma_m$ and a convergence rate $\lambda \in (0,1)$ satisfying the cycle of relations
    \begin{equation*}
        S_1 \xrightarrow{\gamma_1} S_2 \xrightarrow{\gamma_2} S_3 \ \dots \ S_{m} \xrightarrow{\gamma_m} S_1
    \end{equation*}
    such that $\gamma_1 \dots \gamma_m < 1$. Then, there exist constants $c_j \geq 0$ such that $\|S_j\|^\lambda \leq c_j$ for all $j \in \{1, \dots, m\}$.
\end{prop}

Note that, if the conditions of the small gain theorem are satisfied, then $S_j^k \rightarrow 0$ at a geometric rate of $\lambda \in (0,1)$ for all $j \in \{1, \dots, m\}$. Our main result is stated below. 



We use the small gain theorem to prove Theorem \ref{thm:main} as follows. Let 
\begin{align*}
    \tilde{A}^k := A^k - A^*, \ \tilde{X}^k := X^k - X^*
\end{align*}
To show that $A^k \rightarrow A^*$ and $X^k \rightarrow X^*$ with a geometric rate of $\lambda$, it is enough to show that $\|\tilde{A}\|^\lambda$ and $\|\tilde{X}\|^\lambda$ are bounded. To show this, we first prove that there exist non-negative numbers $\gamma_1,\gamma_2,\gamma_3$ such that the cycle of relations
\begin{equation}
    \label{eq:arrows_desired}
    \tilde{A} + \rho \Delta Y \xrightarrow{\gamma_1} X_\perp \xrightarrow{\gamma_2} Y_\perp \xrightarrow{\gamma_3} \tilde{A} + \rho \Delta Y
\end{equation}
holds with $\gamma_1 \gamma_2 \gamma_3 < 1$, where $\Delta Y^k := Y^k - Y^{k-1}, X_\perp^k := X^k - (\onen)X^k, Y_\perp^k := Y^k - (\onen)Y^k$ for all $k \geq 0$.
The proof of each arrow in \eqref{eq:arrows_desired} is provided in Appendix \ref{sec:appendix_arrows}. In particular, we show that the conditions given in Theorem \ref{thm:main} are sufficient to prove the cycle of relations in \eqref{eq:arrows_desired} such that $\gamma_1 \gamma_2 \gamma_3 < 1$. Then, by the small gain theorem (Proposition \ref{prop:small_gain_thm}), it follows that $\|Y_\perp\|^\lambda$ and $\|\tilde{A} + \rho \Delta Y\|^\lambda$ are bounded.

After proving that $\|Y_\perp\|^\lambda$ and $\|\tilde{A} + \rho \Delta Y\|^\lambda$ are bounded, it is shown in Appendix \ref{sec:appendix_add_ineq} that, for some constants $\gamma_a \geq 0,\gamma_x \geq 0,\omega_a, \omega_x$, 
\begin{equation}
    \label{eq:additional_inequalities}
    \begin{aligned}
        \|\tilde{A}\|^\lambda &\leq \gamma_a \|Y_\perp\|^\lambda + \omega_a, \\ \|\tilde{X}\|^\lambda &\leq \gamma_x \|\tilde{A} + \rho \Delta Y\|^\lambda + \omega_x. 
    \end{aligned}
\end{equation}
Then, it follows that $\|\tilde{A}\|^\lambda \leq c_a$ and $\|\tilde{X}\|^\lambda \leq c_x$ for some non-negative constants $c_a$ and $c_x$. From the definition of $\|S\|^\lambda$ given in \eqref{eq:defn_lambda_norm}, this implies $\|\tilde{A}^k\| \leq c_a \lambda^k$ and  $\|\tilde{X}^k\| \leq c_x \lambda^k$
for all $k \geq 0$, which completes the proof of Theorem \ref{thm:main}.


\section{Numerical Examples}
\label{sec:num_examples}

The observations noted in this section were made across many examples. For simplicity, we present the results from one such example.\footnote{The datasets generated during the current study are available from the corresponding author on request.}

\subsection{Sensor network}

\textbf{Problem setup:} We consider a sensor network of $n = 50$ agents. Each agent is placed in the unit square uniformly at random, and it is assigned a random broadcast range between $0.2$ and $0.4$. This generates a directed graph. 
We consider the problem of sensor fusion, where the goal of all the agents is to estimate a common parameter using data from the neighbouring agents. The local objective function of agent $i$ is $f_i(x) = (1/2) \|H_i x - g_i\|^2$, where the measurement matrix $H_i \in \R^{10 \times 2}$ and the data matrix $g_i \in \R^2$ are generated randomly from a standard normal distribution. To estimate the parameter, the agents must solve problem \eqref{eq:opt_prob_intro} in a distributed manner. 

\textbf{Algorithms chosen for comparison:} 
We consider DC-DistADMM \cite{dcdistadmm} and D-ADMM-FTERC \cite{finite-time_admm}. 
Note that the bounds on the parameters of each of these algorithms, such as the ones on $\rho$ and $B$ given in Theorem \ref{thm:main}, are often very conservative. Hence, we tune the parameters to get the best convergence rate, as suggested in \cite{nedich_dig} .

\textbf{Convergence rate:} A plot of the normalized primal residual obtained by each of the algorithms mentioned above is shown in Fig. \ref{fig:basic_comparison} (left). It can be observed that the performance of Algorithm \ref{alg:main} is better than that of D-ADMM-FTERC and DC-DistADMM in this example. 

\begin{figure}
    \centering
    {\includegraphics[width = 0.49\linewidth]{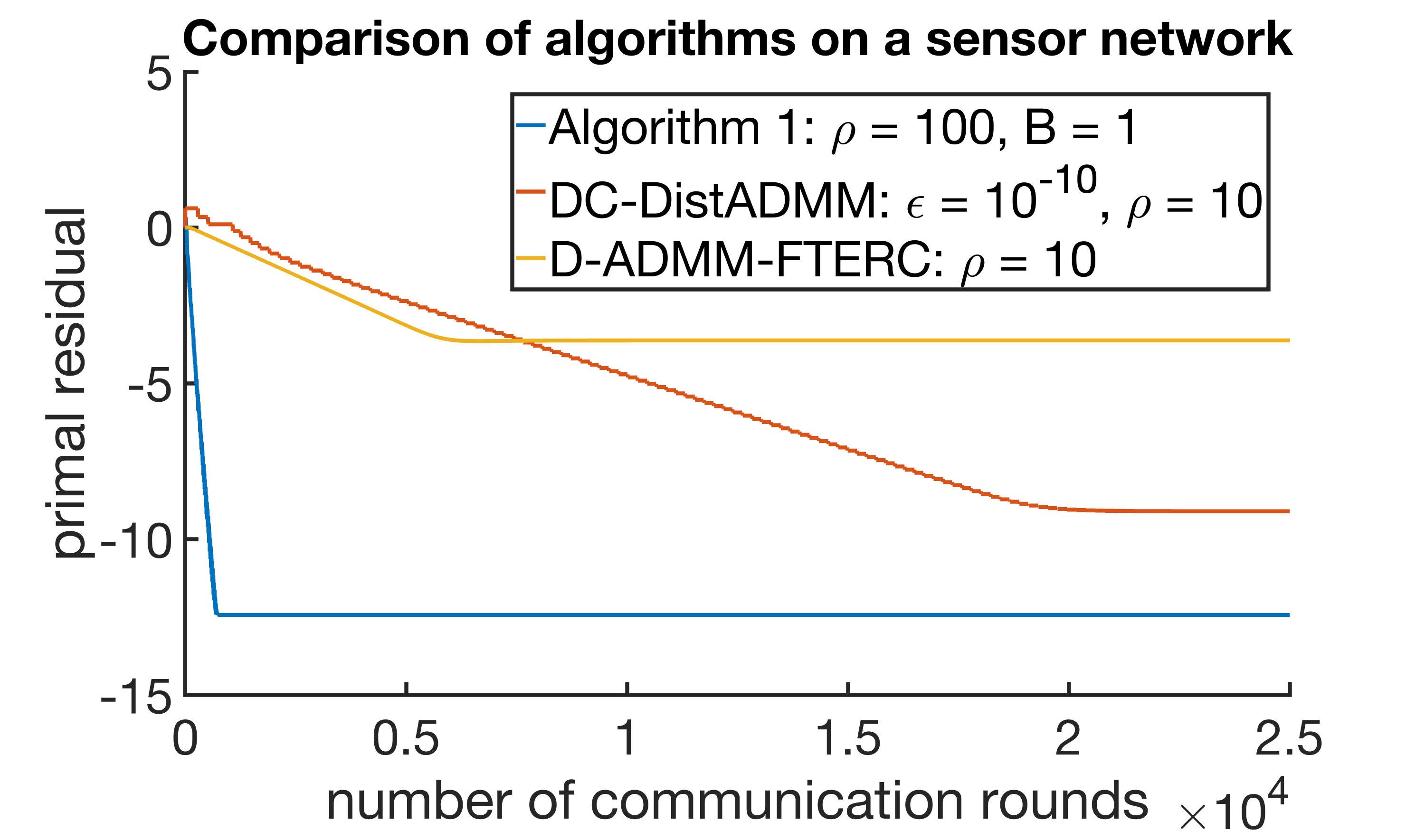}}
    {\includegraphics[width = 0.49\linewidth]{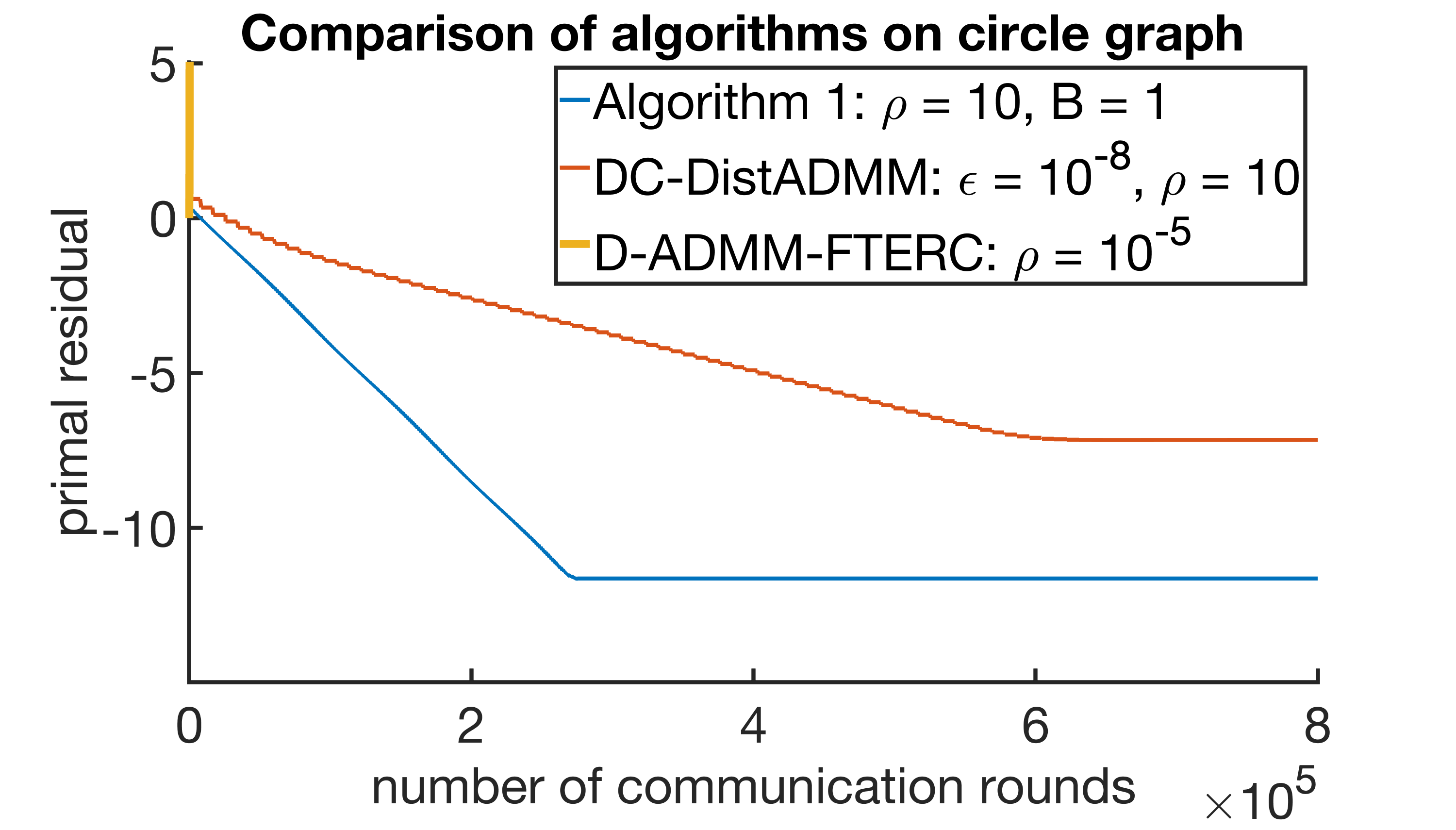}}
    \caption{Plot of $\log(\|X^k - X^*\|/\|X^0 - X^*\|)$ vs. the number of communication rounds for various algorithms on a randomly generated sensor network of $n = 50$ nodes (left), and on a directed circle graph of $n = 50$ nodes (right).}
    \label{fig:basic_comparison}
\end{figure}

It is worth noting that the algorithms mentioned above also generate a sequence $\{A^k\}$ of dual variable iterates which converge to the dual optimal point $A^*$ (not shown here due to space constraints). This is not the case with gradient-based algorithms.

\textbf{Communication cost:} For each algorithm, the number of values sent by an agent per communication round is $O(m)$, where $m$ is the dimension of the unknown parameter.

\subsection{Graphs with large diameters}

Most of the graphs observed in practice are sparse and hence have large diameters. It was observed that Algorithm \ref{alg:main} performs much better than DC-DistADMM and D-ADMM-FTERC on graphs with large diameters, such as a directed circle graph and an undirected line graph. The residuals generated by each algorithm on a directed circle graph of $n = 50$ agents are shown in Fig. \ref{fig:basic_comparison} (right). Note that the residual generated by D-ADMM-FTERC did not converge to zero despite tuning the value of the step size to $10^{-5}$. The intuition behind the superior performance of Algorithm \ref{alg:main} is as follows. It is known that the error in consensus after $k$ steps is upper-bounded by $\eta(W)^k$, where $\eta(W)$ is the second-smallest eigenvalue of the weight matrix $W$ used to achieve consensus \cite{boyd_optimal_W}. Further, $\eta(W)$ has a lower-bound that is inversely proportional to the diameter of the graph \cite{mohar_eigenvalues}. Thus, graphs with larger diameters take more iterations to achieve consensus. Algorithm \ref{alg:main} does not require any accuracy on the inner consensus loop, while the other algorithms require either exact consensus or $\epsilon$-consensus at each step. ADMM appears to be robust to errors in consensus. Hence, enforcing consensus accuracy at each iteration may be an overkill.


\subsection{Sensitivity of Algorithm \ref{alg:main} to its parameters}

\begin{figure}
    \centering
    {\includegraphics[width=0.49\linewidth]{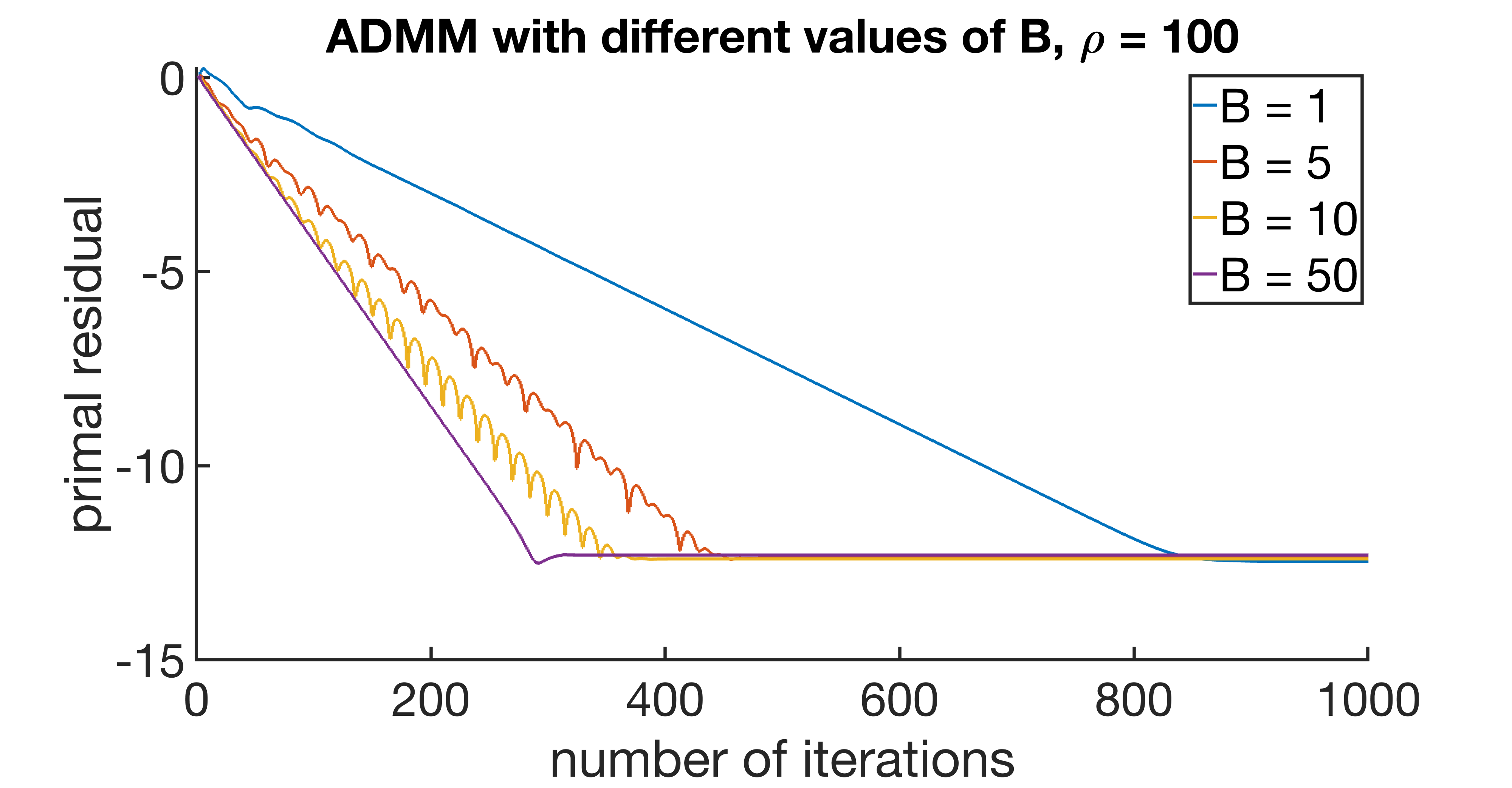}}
    {\includegraphics[width=0.49\linewidth]{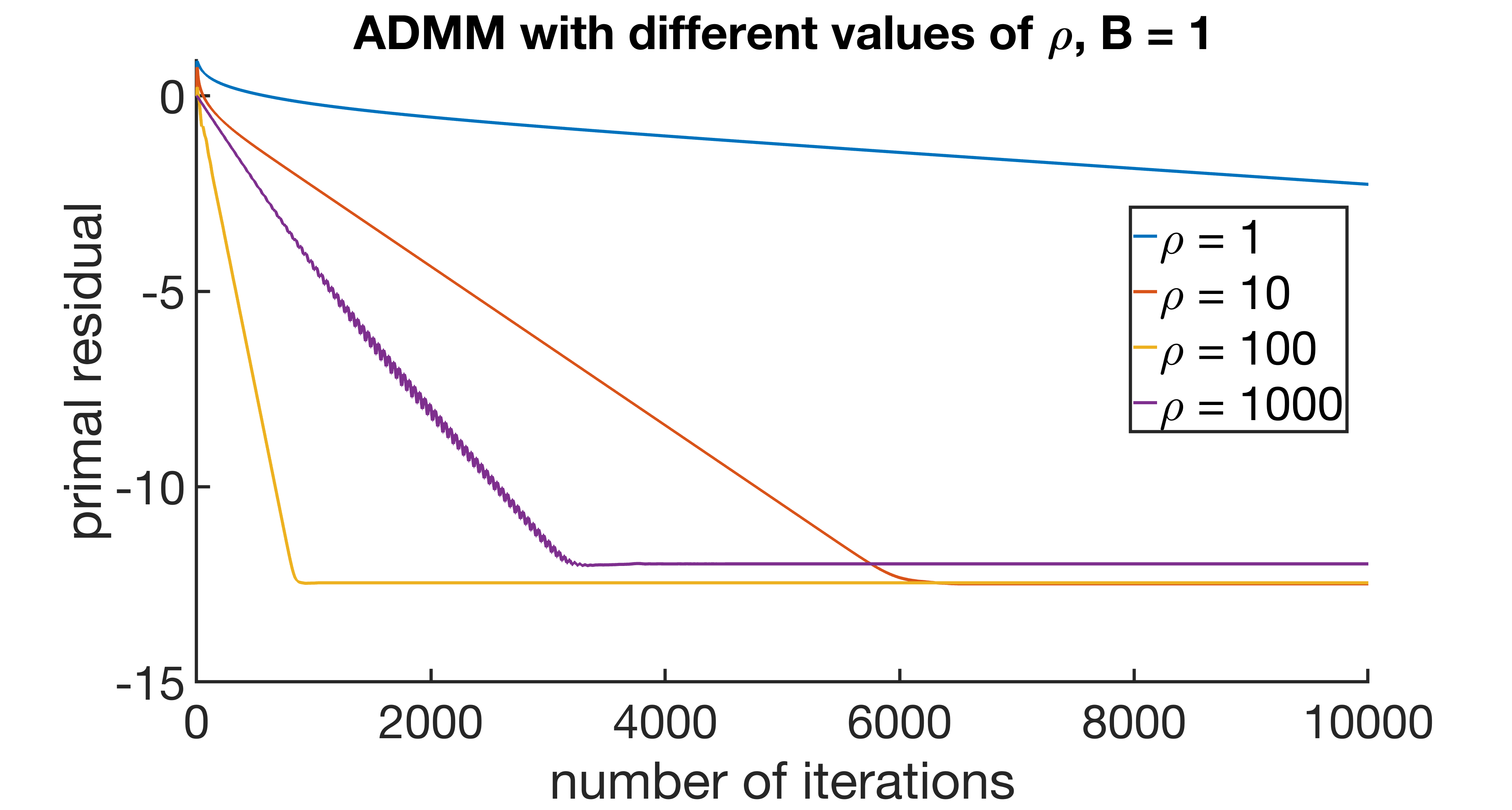}}
    \caption{Plot of $\log(\|X^k - X^*\|/\|X^0 - X^*\|)$ vs. the number of iterations for Algorithm \ref{alg:main} with $\rho = 100$ and various values of $B$ (left), and with $B = 1$ and various values of $\rho$ (right).}
    \label{fig:admm_various_Bs_rhos}
\end{figure}



\textbf{Effect of $B$:} We study the effect of changing the number of communication rounds per iteration ($B$) in the inner loop of our ADMM algorithm. As observed in Fig. \ref{fig:admm_various_Bs_rhos} (left), with an increase in $B$, the number of iterations required to reach the optimal value decreases, while there is no change in the final value.


\textbf{Robustness to changes in $\rho$:} 
We observe in  Fig. \ref{fig:admm_various_Bs_rhos} (right), that the primal residual decreases monotonically for a large range of values of $\rho$, albeit with different rates. On the other hand, primal-descent algorithms are known to be sensitive to their step-size. 

\section{Conclusion}
\label{sec:conclusion}

We proposed an ADMM algorithm to solve distributed optimization problems over directed graphs. Our algorithm uses the ideas of balancing weights and dynamic average consensus. Under the assumption that the objective function is strongly convex and smooth, we showed that the primal-dual iterates of the algorithm converge to their unique optimal points at a geometric rate, provided the parameters of the algorithm are chosen appropriately. 
Through a numerical example, we demonstrated that 
our algorithm gives a better performance than some state-of-the-art ADMM methods over directed graphs. Additionally, the algorithm was observed to be robust to changes in its parameters. In the future, it will be interesting to see if convergence of the algorithm can be guaranteed by relaxing the assumptions of strong convexity and smoothness. Further, it will be interesting to extend the algorithm to time-varying graphs.




\bibliographystyle{plain}        
\bibliography{references}           

\section*{APPENDIX}

\section{Proof of Lemma \ref{lem:y_update}}
\label{sec:appendix_proof_of_y_update_lemma}

The update step in \eqref{eq:y_update_original} can be written as
\begin{equation}
    \label{eq:y_update_distance_form}
    Y^{k+1} = \argmin_{Y = (\onen)Y} \|Y - (X^{k+1} + A^k/\rho)\|^2.
\end{equation}
We parameterize the constraint set of \eqref{eq:y_update_distance_form}, i.e., $\mathcal{C} = \{Y \in \R^{n \times m}: Y = (\onen)Y\}$, using the parameter $\nu = (\nu_1,\dots,\nu_m) \in \R^m$ as $\mathcal{C} = \{Y(\nu) \in \R^{n \times m}: Y(\nu) = (\onen)(X^{k+1} + A^k/\rho) + (\one_n \one_m^T/n)\diag(\nu)\}$. With this parameterization, let $Y(\nu^*)$ be the minimizer of \eqref{eq:y_update_distance_form}. We show that $\nu^* = 0$ as follows. By definition,


\begin{align*}
    Y^{k+1} &= Y(\nu^*) \\
    &= \argmin_{Y(\nu) = (\onen)(X^{k+1} + A^k/\rho) + (\one_n \one_m^T/n)\diag(\nu)} \\
    &\hspace{4cm}\|Y(\nu) - (X^{k+1} + A^k/\rho)\|^2.
\end{align*}
This implies
\begin{align*}
    \nu^* &= \argmin_{\nu \in \R^m} \|(I - \onen)(X^{k+1} + A^k/\rho) \\ 
    &\hspace{4cm}- (\one_n \one_m^T/n)\diag(\nu)\|^2 \\
    &= \argmin_{\nu \in \R^m} \|(I - \onen)(X^{k+1} + A^k/\rho)\|^2 + n\|\nu\|^2 \\
    &= 0,
\end{align*}
where we have used the fact that the matrices $(I - \onen)(X^{k+1} + A^k/\rho)$ and $(\one_n \one_m^T/n)\diag(\nu)$ are orthogonal. Thus, we have $Y^{k+1} = Y(\nu^*) = Y(0) = (\onen)(X^{k+1} + A^k/\rho)$, i.e., $y_i^{k+1} = (1/n) \sum_{j = 1}^n (x_j^{k+1} + a_j^k/\rho)$ for all $i \in V$.

\section{Proof of Lemma \ref{lem:W_properties}}
\label{sec:appendix_proof_of_W_lemma}

(1) The statement follows from the fact that, by definition, $W^k(b) = I - (\Dout - A)\diag(w^k(b))$ is column stochastic. 

(2) It is shown in \cite[Lemma 1]{makhdoumi_bal_weights} that $\lim_{k \rightarrow \infty} P^k$ exists. Hence, $\lim_{k \rightarrow \infty} w^k(b) = P^{b + (k-1)(B-1)}w^1(0)$ exists and is independent of $b$. This implies $\lim_{k \rightarrow \infty} W^k(b)$ exists and is independent of $b$. Let $W^\infty := \lim_{k \rightarrow \infty} W^k(b)$. Now, by the triangle inequality,
\begin{align}
\label{eq:Wkb_triangle_ineq}
    &\|W^k(b) - p^k(b) \one^T\| \nonumber \\
    &\leq \|W^k(b) - W^\infty\|
    + \|W^\infty - \one \one^T/n\| \nonumber \\
    & \quad + \|\one \one^T/n - p^k(b) \one^T\|.
\end{align}
We analyze each term in the RHS of \eqref{eq:Wkb_triangle_ineq} one-by-one. 

(a) $\|W^k(b) - W^\infty \|$: By definition of $W^\infty$, $\|W^k(b) - W^\infty\| \rightarrow 0$ as $k \rightarrow \infty$. 

(b) $\|W^\infty - \one \one^T/n\|$: 
It is shown in \cite[Lemma 1]{makhdoumi_bal_weights} that $W^\infty$ is a doubly-stochastic matrix. Further, it is shown in \cite[Lemma 2]{makhdoumi_bal_weights} that $W^\infty$ is a non-negative matrix with $[W^\infty]_{ii} > 0$ for all $i \in V$ and $[W^\infty]_{ij} > 0$ if and only if $(i,j) \in E$. Using these facts, we first show that $W^\infty (W^\infty)^T$ is a primitive matrix. To show this, it is enough to show that $\big(W^\infty (W^\infty)^T\big)^D$ is a positive matrix, where $D$ is the diameter of $G$. For any $(i,j) \in V \times V$,
\begin{align*}
    &\big[(W^\infty)^D\big]_{ij} \\
    &= \sum_{l_1 \in V} \sum_{l_2 \in V} \dots \sum_{l_D \in V} [W^\infty]_{il_1} [W^\infty]_{l_1l_2} \dots [W^\infty]_{l_D j}.
\end{align*}
By definition, $D$ is the length of the largest path connecting any two nodes. Moreover, $G$ is strongly connected. Hence, any given $i,j \in V$, there exists a path $(q_1,i), (q_2,q_1), \dots, (j,q_D)$ from $i$ to $j$. Hence, $$\left[(W^\infty)^D\right]_{ij} \geq [W^\infty]_{iq_1} [W^\infty]_{q_1q_2} \dots [W^\infty]_{q_D j} > 0.$$ Thus, $(W^\infty)^D$ is a positive matrix. Now, $[W^\infty (W^\infty)^T]_{ij} \geq [W^\infty]_{ii} [W^\infty]_{ji} > 0$. Thus, $\left(W^\infty (W^\infty)^T\right)^D$ is also a positive matrix. It then follows that $W^\infty (W^\infty)^T$ is a primitive matrix. This implies $\rho\big(W^\infty (W^\infty)^T - \one \one^T/n\big) < 1$ \cite[(8.3.16)]{meyer}. Now,
\begin{align*}
    &\|W^\infty - \one \one^T/n\|^2 \\
    &= \rho\big((W^\infty - \one \one^T/n)(W^\infty - \one \one^T/n)^T\big) \\
    &= \rho\big(W^\infty (W^\infty)^T - \one \one^T/n\big) \\
    &< 1,
\end{align*}
where we have used the facts $W^\infty \one = \one, \one^T W^\infty = \one^T$ as shown in \cite[Lemma 1]{makhdoumi_bal_weights}. Thus, $\|W^\infty - \one \one^T/n\| < 1$.

(c) $\|\one \one^T/n - p^k(b) \one^T\|$: Consider any $\epsilon > 0$. We show that $\exists k_1 \geq 0$ such that $\|\one \one^T/n - p^k(b) \one^T\| < \epsilon$ for all $k \geq k_1$. For all $k \geq 0$ and $l \geq 0$,
\begin{align}
\label{eq:one_minus_p^k(b)_triangle_ineq}
    &\|\one \one^T/n - p^k(b) \one^T\| \nonumber \\
    &\leq \|\one \one^T/n - (W^\infty)^l\| + \|(W^\infty)^l - \big(W^k(b)\big)^l\| \nonumber \\ 
    & \quad + \|\big(W^k(b)\big)^l - p^k(b) \one^T\|.
\end{align}
We argued above that $W^\infty$ is a primitive matrix with $\one^T W^\infty = \one^T$ and $W^\infty \one = \one$. Hence, $\rho(W^\infty - \one \one^T/n) < 1$. It follows that $\exists l_1 \geq 0$ such that $\|(W^\infty)^l - \one \one^T /n\| \leq \epsilon/3$ for all $l \geq l_1$ \cite[(8.3.10)]{meyer}. Following the same arguments as with $W^\infty$ above, we can show that $\rho\big(W^k(b) - p^k(b) \one^T\big) < 1$ for all $k \geq 0$ and $0 \leq b \leq B-1$. Hence, for all $k \geq 0$, $\exists l_2 \geq 0$ such that $\|\big(W^k(b)\big)^l - p^k(b) \one^T\| \leq \epsilon/3$ for all $l \geq l_2$. Moreover, since $\lim_{k \rightarrow \infty} W^k(b) = W^\infty$, the eigenvalues of $W^k(b)$ converge to those of $W^\infty$. This implies $l_2$ is independent of $k$ (see \cite[(8.3.10)]{meyer}, which implies that the rate of convergence of $\big(W^k(b)\big)^l$ to $p^k(b) \one^T$ depends only on the second largest eigenvalue $W^k(b)$). Now, let $l_0 := \max\{l_1,l_2\}$. Note that given any matrix $C \in \R^{n \times n}$, $C^{l_0}$ is a continuous function of the entries of $C$. Hence, $\lim_{k \rightarrow \infty} W^k(b) = W^\infty$ implies $\exists k_1 \geq 0$ such that $\|(W^\infty)^{l_0} - \big(W^k(b)\big)^{l_0}\| \leq \epsilon/3$ for all $k \geq k_1$. Substituting the aforementioned facts in \eqref{eq:one_minus_p^k(b)_triangle_ineq}, we have $\|\one \one^T/n - p^k(b) \one^T\| \leq \epsilon$ for all $k \geq k_1$. This implies $\|\one \one^T/n - p^k(b) \one^T\| \rightarrow 0$ as $k \rightarrow \infty$.

Using the results derived for each of the three terms above, from \eqref{eq:Wkb_triangle_ineq} we can conclude that $\|W^k(b) - p^k(b) \one^T\| < 1$ for $k$ large-enough. 

(3) The statement follows from the fact that $\|\one \one^T/n - p^k(b) \one^T\| \rightarrow 0$ as $k \rightarrow \infty$, which we have proved above. 

(4) For all $k \geq 1$, $\|W^k(b)\| \leq \|W^k(b) - p^k(b)\one^T\| + \|p^k(b)\one^T\|$. Now, choosing $\hat{k} \geq \bar{k}$ implies $\|W^k(b) - p^k(b)\one^T\| \leq \delta$ for all $k \geq \hat{k}$. On the other hand, $\|p^k(b)\one^T\|$ is bounded since $p^k(b)$ converges to $\one/n$. Thus, $\|W^k(b)\| \leq M$ for all $k \geq \hat{k}$ for some $M \geq 0$.

\section{Some intermediate results}

We state some useful intermediate results. 

\begin{lem}
\label{lem:square_root_sequences_inequality}
    Let $\{p^k\}$, $\{q_i^k\}, i \in \{1, \dots, m\}$, be sequences in $\R$ such that $q_i^k \geq 0$ for all $i,k$. Suppose $(p^k)^2 \leq \sum_{i = 1}^m (q_i^k)^2 + c$
    for some constant $c \geq 0$, for all $k$. Then, $p^k \leq \sum_{i = 1}^m q_i^k + \sqrt{c}$ for all $k$.
\end{lem}
\begin{proof}
Since $\{q_i^k\}$ is a non-negative sequence for all $i \in \{1, \dots, m\}$ and $c \geq 0$, we have
\begin{align*}
    \left(\sum_{i = 1}^m q_i^k + \sqrt{c}\right)^2 \geq \sum_{i = 1}^m (q_i^k)^2 + c
\end{align*}
for all $k$. Now, using $(p^k)^2 \leq \sum_{i = 1}^m (q_i^k)^2 + c$, we obtain
\begin{align*}
    \left(\sum_{i = 1}^m q_i^k + \sqrt{c}\right)^2 \geq \sum_{i = 1}^m (q_i^k)^2 + c \geq (p^k)^2.
\end{align*}
\end{proof}

Let $f^*$ denote the conjugate of $f$, i.e., $f^*(a) = \sup_X \ip{A}{X} - f(X)$ for all $a \in \R^n$ such that the supremum is finite.
\begin{lem}
    \label{lem:dual_opt_point}
    The dual optimal point $A^*$ of \eqref{eq:opt_prob_indicator_fn} satisfies $(\onen)A^* = 0$.
\end{lem}
\begin{proof}
Given the standard Lagrangian $L_0$ of \eqref{eq:opt_prob_indicator_fn} (equation \eqref{eq:aug_lagrangian} with $\rho = 0$), the dual problem of \eqref{eq:opt_prob_indicator_fn} is given by
\begin{align*}
    &\sup_A \inf_{(X,Y)} L_0(X,Y,A) \\
    &=\sup_A \Big(\inf_X (f(X) + \ip{A}{X}) + \inf_Y (I(Y) - \ip{A}{Y}) \Big) \\
    &=\sup_A \Big(-f^*(-A) + \inf_{Y = (\onen)Y} -\ip{A}{Y}\Big) \\
    &=\sup_{(\onen)A = 0} -f^*(-A).
\end{align*}
Thus, the dual optimal point satisfies $(\onen)A^* = 0$.
\end{proof}

\begin{prop}{\cite[Theorem 6]{kakade_duality}}
    \label{prop:f^*_properties}
    Under Assumption \ref{assum:strong_conv_lip_grad}, $f^*$ is $1/L$-strongly convex and $1/\mu$-smooth.
\end{prop}

\begin{lem}
    \label{lem:grad_f^*_and_x_relation}
    The primal-dual iterates of Algorithm \ref{alg:main} satisfy $X^{k+1} = \nabla f^*(-A^{k+1}-\rho \Delta Y^{k+1})$ for all $k \geq 0$ and the primal-dual optimal points of \eqref{eq:opt_prob_indicator_fn} satisfy $X^* = \nabla f^*(-A^*)$.
\end{lem}
\begin{proof}
From the definition of $f^*$, we have
\begin{align*}
    f^*(-A) = \sup_{x} -\ip{A}{X} - f(X).
\end{align*}
This implies
\begin{align}
    \label{eq:grad_f^*_defn}
    \nabla f^*(-A) &= \argmax_X -\ip{A}{X} - f(X) \nonumber \\
    &= \argmin_X \ip{A}{X} + f(X).
\end{align}
Now, from the update step \eqref{eq:x_update_explicit} of Algorithm \ref{alg:main}, we have, for all $k \geq 0$,
\begin{align*}
    &\nabla f(X^{k+1}) + A^k + \rho(X^{k+1} - Y^k) = 0.
\end{align*}
Using \eqref{eq:a_update}, this implies
\begin{align*}
    &\nabla f(X^{k+1}) + A^{k+1} + \rho \Delta Y^{k+1} = 0.
\end{align*}
Hence, we have
\begin{align*}
    X^{k+1} = \argmin_X f(X) + \ip{A^{k+1} + \rho\Delta Y^{k+1}}{X}.
\end{align*}
Now, using above relation with \eqref{eq:grad_f^*_defn} gives $\nabla f^*(-A^{k+1}-\rho \Delta Y^{k+1}) = X^{k+1}$, which is the first part of the lemma. Next, by definition of a primal-dual optimal point, we have
\begin{align*}
    (X^*,Y^*) = \argmin_{(X,Y)} L_0(X,Y,A^*),
\end{align*}
where $L_0$ is the Lagrangian of \eqref{eq:opt_prob_indicator_fn} (equation \eqref{eq:aug_lagrangian} with $\rho = 0$). This implies $X^* = \argmin_X f(X) + \ip{A^*}{X}$. Using this relation with \eqref{eq:grad_f^*_defn}, it follows that $\nabla f^*(-A^*) = X^*$. 
\end{proof}

\section{Proof of each arrow in \eqref{eq:arrows_desired}}
\label{sec:appendix_arrows}

\textbf{Proof of the first arrow ($\tilde{A} + \rho \Delta Y \xrightarrow{\gamma_1} X_\perp$):}
From Proposition \ref{prop:f^*_properties}, we obtain that for all $k \geq 0$,
\begin{align*}
    &\|\nabla f^*(-A^{k+1}-\rho \Delta Y^{k+1}) - \nabla f^*(-A^*)\| \\
    &\leq \frac{1}{\mu} \|A^{k+1} - A^* + \rho \Delta Y^{k+1}\| = \frac{1}{\mu} \|\tilde{A}^{k+1} + \rho \Delta Y^{k+1}\|.
\end{align*}
Using Lemma \ref{lem:grad_f^*_and_x_relation} with the relation above, we have
\begin{align}
    \label{eq:bound_x_tilde_^k+1_square}
    \|X^{k+1} - X^*\|^2 \leq \frac{1}{\mu^2}\|\tilde{A}^{k+1} + \rho \Delta Y^{k+1}\|^2.
\end{align}
Further, $\|X^{k+1} - X^*\|^2 = \|X_\perp^{k+1}\|^2 + \|(\onen) X^{k+1} - X^*\|^2$. Substituting this relation in \eqref{eq:bound_x_tilde_^k+1_square}, we have
\begin{equation}
\label{eq:x_perp_k+1_bound}
    \|X_\perp^{k+1}\| \leq \frac{1}{\mu} \|\tilde{A}^{k+1} + \rho \Delta Y^{k+1}\|.
\end{equation}
Hence,
\begin{equation}
\label{eq:x_perp_lamba_K_bound}
    \|X_\perp\|^{\lambda,K} \leq \frac{1}{\mu} \|\tilde{A} + \rho \Delta Y\|^{\lambda,K}  + \|\tilde{X}^0\|
\end{equation}
for all $K \geq 0$. Thus, $\gamma_1 := 1/\mu$.

\textbf{Proof of the second arrow ($X_\perp \xrightarrow{\gamma_2} Y_\perp$):}
Using the update rule \eqref{eq:zeta_update} of $\zeta^{k}(b)$ and the first property in Lemma \ref{lem:W_properties}, we have, for all $k \geq 1$ and $b \in \{0, \dots, B-1\}$,
\begin{align*}
    &\big(I - p^k(b) \one^T\big)\zeta^k(b+1) \\
    &= \big(I - p^k(b) \one^T\big)W^k(b) \zeta^k(b) \\
    &= \big(W^k(b) - p^k(b) \one^T\big) \zeta^k(b) \\
    &= \big(W^k(b) - p^k(b) \one^T - p^k(b) \one^T + p^k(b) \one^T\big) \zeta^k(b) \\
    &= \big(W^k(b) - p^k(b) \one^T - W^k(b) p^k(b) \one^T \\
    &\hspace{30ex} + p^k(b) \one^T p^k(b) \one^T\big) \zeta^k(b) \\
    &= \big(W^k(b) - p^k(b) \one^T\big) \big(I - p^k(b) \one^T\big) \zeta^k(b).
\end{align*}
Hence,
\begin{align*}
    &\big(I - p^k(b) \one^T + \one \one^T/n - \one \one^T/n\big)\zeta^k(b+1) \\
    &= \big(W^k(b) - p^k(b) \one^T\big) \times \\
    & \quad \quad \big(I - p^k(b) \one^T + \one \one^T/n - \one \one^T/n\big) \zeta^k(b).
\end{align*}
This implies
\begin{align*}
    &\zeta_\perp^k(b+1) + \big(\one/n - p^k(b)\big) \one^T \zeta^k(b+1) \\
    &= \big(W^k(b) - p^k(b) \one^T\big) \Big[\zeta_\perp^k(b) + \big(\one/n - p^k(b)\big) \one^T \zeta^k(b)\Big],
\end{align*}
where $\zeta_\perp^k(b) := (I - \one \one^T/n) \zeta^k(b)$.
This implies
\begin{align*}
    &\|\zeta_\perp^k(b+1) + \big(\one/n - p^k(b)\big) \one^T \zeta^k(b+1)\| \\
    &\leq \|W^k(b) - p^k(b) \one^T\|_2 \|\zeta_\perp^k(b) + \big(\one/n - p^k(b)\big) \one^T \zeta^k(b)\|,
\end{align*}
Since $\zeta^k(0) = X^k$ and $\one^T W^k(b) = \one^T$, we have $\one^T \zeta^k(b+1) = \one^T W^k(b) \zeta^k(b) = \one^T \zeta^k(b) = \one^T X^k$. Further, by the second property in Lemma \ref{lem:W_properties}, $\exists \bar{k} \geq 0$, $\exists \delta \in [0,1)$ such that for all $k \geq \bar{k}$, $\|W^k(b) - p^k(b) \one^T\| \leq \delta$. Hence, for all $k \geq \bar{k}$,
\begin{align*}
    &\|\zeta_\perp^k(b+1) + \big(\one/n - p^k(b)\big) \one^T X^k\| \\
    &\leq \delta \|\zeta_\perp^k(b) + \big(\one/n - p^k(b)\big) \one^T X^k\|,
\end{align*}
Now, by the third property in Lemma \ref{lem:W_properties}, we know that $\lim_{k \rightarrow \infty} \|\one/n - p^k(b)\| = 0$. Since $\|\cdot\|$ is a continuous function, there must exist a $\tilde{k} \geq \bar{k}$ such that for all $k \geq \tilde{k}$, $\|\zeta_\perp^k(b+1)\| \leq \delta \|\zeta_\perp^k(b)\|$. Using this inequality iteratively from $b = 0$ to $b = B-1$, we have $\|\zeta_\perp^k(B)\| \leq \delta \|\zeta_\perp^k(0)\|$ for all $k \geq \tilde{k}$. By definition, $\zeta^k(B) = Y^k$ and $\zeta^k(0) = X^k$. Hence, the last inequality is equivalent to $\|Y_\perp^{k}\| \leq \delta^B \|X_\perp^{k}\|$
for all $k \geq \tilde{k}$. Hence, for all $K \geq 0$,
\begin{align*}
    \|Y_\perp\|^{\lambda,K} \leq \delta^B \|X_\perp\|^{\lambda,K} + c_1
\end{align*}
for some constant $c_1 \geq 0$. Thus, $\gamma_2 := \delta^B$.

\textbf{Proof of the third arrow ($Y_\perp \xrightarrow{\gamma_3} \tilde{A} + \rho \Delta Y$):}
We prove the third arrow in two steps. In the first step, we bound $\Delta Y$ by a linear combination of $Y_\perp$ and $\tilde{A}$. In the second step, we bound $\tilde{A}$ by a linear combination of $Y_\perp$ and $\Delta Y$. To complete the proof, we use the two bounds together.

\emph{Step 1:} Using triangle inequality,
\begin{align*}
    &\|\Delta Y^{k+1}\| \\
    &= \|Y^{k+1} - Y^k\| \\
    &\leq\|Y^{k+1} - (\one \one^T/n) Y^{k+1}\| + \|Y^k - (\one \one^T/n) Y^k\| \\
    & \quad + \|(\one \one^T/n) Y^{k+1} - (\one \one^T/n) Y^k\| \\
    &= \|Y_\perp^{k+1}\| + \|Y_\perp^{k}\| + \|(\one \one^T/n) Y^{k+1} - (\one \one^T/n) Y^k\|.
\end{align*}
We know that $\one^T Y^k = \one^T X^k$ for all $k \geq 0$. Hence,
\begin{align*}
    &\|\Delta Y^{k+1}\| \\
    &\leq \|Y_\perp^{k+1}\| + \|Y_\perp^{k}\| + \|(\one \one^T/n) X^{k+1} - (\one \one^T/n) X^k\| \\
    &\leq \|Y_\perp^{k+1}\| + \|Y_\perp^{k}\| + \|X_\perp^{k+1}\| + \|X_\perp^{k}\| + \|X^{k+1} - X^k\|.
\end{align*}
From \eqref{eq:x_perp_k+1_bound}, we know that $\|X_\perp^{k+1}\| \leq (1/\mu) \|\tilde{A}^{k+1} + \rho \Delta Y^{k+1}\|$ for all $k \geq 0$. Using this with Lemma \ref{lem:grad_f^*_and_x_relation} and Proposition \ref{prop:f^*_properties}, we have
\begin{align*}
    &\|X^{k+1} - X^k\| \\
    &\leq \|X^{k+1} - X^*\| + \|X^k - X^*\| \\
    &= \|\nabla f^*(-A^{k+1}-\rho \Delta Y^{k+1}) - \nabla f^*(-A^*)\| + \\
    & \quad \|\nabla f^*(-A^k-\rho \Delta Y^k) - \nabla f^*(-A^*)\| \\
    &\leq \frac{1}{\mu} \big(\|A^{k+1} - A^* + \rho \Delta Y^{k+1}\| + \|A^k - A^* + \rho \Delta Y^k\|\big) \\
    &\leq \frac{1}{\mu} \big(\|\tilde{A}^{k+1}\| + \|\tilde{A}^k\| + \rho(\|\Delta Y^{k+1}\| + \|\Delta Y^k\|)\big).
\end{align*}
Hence, 
\begin{align*}
    &\|\Delta Y^{k+1}\| \leq \|Y_\perp^{k+1}\| + \|Y_\perp^{k}\| + \\
    &\quad \frac{2}{\mu} (\|\tilde{A}^{k+1}\| + \|\tilde{A}^{k}\| + \rho(\|\Delta Y^{k+1}\| + \|\Delta Y^k\|)).
\end{align*}
Dividing by $\lambda^{k+1}$ on both sides, we have
\begin{align*}
    &\frac{\|\Delta Y^{k+1}\|}{\lambda^{k+1}} \leq \frac{\|Y_\perp^{k+1}\|}{\lambda^{k+1}} + \frac{1}{\lambda}\frac{\|Y_\perp^{k}\|}{\lambda^k} \\
    & \quad+ \frac{2}{\mu} \Bigg(\frac{\|\tilde{A}^{k+1}\|}{\lambda^{k+1}} + \frac{1}{\lambda}\frac{\|\tilde{A}^{k}\|}{\lambda^{k}} + \rho\bigg(\frac{\|\Delta Y^{k+1}\|}{\lambda^{k+1}} + \frac{1}{\lambda}\frac{\|\Delta Y^k\|}{\lambda^k}\bigg)\Bigg)
\end{align*}
for all $k \geq 0$.
Taking supremum over $k = 0, 1, \dots, K-1$, we have
\begin{align*}
    &\|\Delta Y\|^{\lambda,K} \leq \left(1 + \frac{1}{\lambda}\right) \|Y_\perp\|^{\lambda,K} \\
    &\quad + \frac{2}{\mu} \Bigg(\left(1 + \frac{1}{\lambda}\right) (\|\tilde{A}\|^{\lambda,K} + \rho \|\Delta Y\|^{\lambda,K})\Bigg) + c_2
\end{align*}
for all $K \geq 0$ for some constant $c_2 \geq 0$. Rearranging the terms, we have
\begin{align*}
    &\left(\frac{\lambda \mu - 2 \rho (1 + \lambda)}{\lambda \mu}\right)\|\Delta Y\|^{\lambda,K} \\
    &\leq \left(\frac{1+\lambda}{\lambda}\right) \|Y_\perp\|^{\lambda,K} + \frac{2}{\mu} \left(\frac{1 + \lambda}{\lambda}\right)\|\tilde{A}\|^{\lambda,K} + c_2.
\end{align*}
From \eqref{eq:rho_condition_1}, we know that $\lambda \mu - 2 \rho (1 + \lambda) > 0$. Hence,
\begin{align}
\label{eq:delta_y_bound}
    &\|\Delta Y\|^{\lambda,K} \leq \left(\frac{\mu(1+\lambda)}{\lambda \mu - 2 \rho (1 + \lambda)}\right) \|Y_\perp\|^{\lambda,K} \nonumber \\
    & \quad + \frac{2}{\mu} \left(\frac{\mu(1 + \lambda)}{\lambda \mu - 2 \rho (1 + \lambda)}\right)\|\tilde{A}\|^{\lambda,K} + c_2.
\end{align}

\emph{Step 2:} Let $\Delta A^{k+1} = A^{k+1} - A^k$. For all $k \geq 0$,
\begin{align*}
    &\|\tilde{A}^k\|^2 = \|A^k - A^*\|^2 \\
    &= \|A^k - A^{k+1} + A^{k+1} - A^*\|^2 \\
    &= \|\Delta A^{k+1}\|^2 + \|\tilde{A}^{k+1}\|^2 + 2\ip{A^k - A^{k+1}}{A^{k+1} - A^*} \\
    &\geq \|\tilde{A}^{k+1}\|^2 + 2\ip{A^k - A^{k+1}}{A^{k+1} - A^*}.
\end{align*}
This implies
\begin{align}
    \label{eq:bound_a_tilde^k+1}
    \|\tilde{A}^{k+1}\|^2 \leq &\|\tilde{A}^k\|^2 + 2\ip{A^{k+1} - A^k}{A^{k+1} - A^*}.
\end{align}
We try to bound the term $2\ip{A^{k+1} - A^k}{A^{k+1} - A^*}$ above. The $A$ update step in Algorithm \ref{alg:main} can be written as
\begin{align*}
    A^{k+1} &= \argmin_{A:(\onen)A = 0} \rho \ip{Y^{k+1} - X^{k+1}}{A} + \frac{1}{2}\|A - A^k\|^2 \\
    &= \argmin_{A:(\onen)A = 0} \rho\ip{Y_\perp^{k+1} - X^{k+1}}{A} + \frac{1}{2}\|A - A^k\|^2,
\end{align*}
From Lemma \ref{lem:dual_opt_point}, we know that $(\onen)A^* = 0$. Evaluating the optimality condition for the minimization problem above at $A = A^*$, we obtain
\begin{align*}
    \ip{\rho(Y_\perp^{k+1} - X^{k+1}) + A^{k+1} - A^k}{A^* - A^{k+1}} \geq 0.
\end{align*}
Using Lemma \ref{lem:grad_f^*_and_x_relation}, this implies
\begin{align}
    \label{eq:bound_(A^k+1-A^k)^T(A^k+1-A^*)}
    &\ip{A^{k+1} - A^k}{A^{k+1} - A^*} \nonumber \\
    &\leq \rho\ip{\nabla f^*(-A^{k+1}-\rho \Delta Y^{k+1}) - Y_\perp^{k+1}}{A^{k+1} - A^*}
\end{align}
For the ease of notation, let $\hat{A} = -A^{k+1}-\rho \Delta Y^{k+1}$. Now, we bound the term $\rho\ip{\nabla f^*(\hat{A}) - Y_\perp^{k+1}}{A^{k+1} - A^*}$. Since $f^*$ is $1/L$-strongly convex, we have, for all $A_1$ in the domain of $f^*$, $f^*(A_1) \geq f^*(\hat{A}) + \ip{\nabla f^*(\hat{A})}{A_1 - \hat{A}} + (1/2L) \|A_1 - \hat{A}\|^2$. This implies
\begin{align*}
    &\ip{\nabla f^*(\hat{A}) - Y_\perp^{k+1}}{A_1 - \hat{A}} \\ &\leq f^*(A_1) - f^*(\hat{A}) - \ip{Y_\perp^{k+1}} {A_1 - \hat{A}} - \frac{\|A_1 - \hat{A}\|^2}{2L} \\
    &\leq f^*(A_1) - f^*(\hat{A}) + \frac{\alpha \|Y_\perp^{k+1}\|^2}{2} \\
    & \hspace{3cm} + \frac{\|A_1 - \hat{A}\|^2}{2\alpha} - \frac{\|A_1 - \hat{A}\|^2}{2L},
\end{align*}
where the last inequality follows by using Peter-Paul inequality on $-\ip{Y_\perp^{k+1}}{A_1 - \hat{A}}$, where $\alpha > 0$ is arbitrary. Similarly, using the fact that $f^*$ is $1/\mu$-smooth, we have, for all $A_2$ in the domain of $f^*$,
\begin{align*}
    &\ip{\nabla f^*(\hat{A}) - Y_\perp^{k+1}}{\hat{A} - A_2} \\
    &\leq f^*(\hat{A}) - f^*(A_2) + \frac{\beta \|Y_\perp^{k+1}\|^2}{2} \\
    &\hspace{3cm} + \frac{\|A_2 - \hat{A}\|^2}{2\beta} + \frac{\|A_2 - \hat{A}\|^2}{2\mu},
\end{align*}
where $\beta > 0$ is arbitrary. Adding the two inequalities above with $A_1 = -A^*, A_2 = -A^{k+1}$, we have
\begin{align*}
    &\ip{\nabla f^*(\hat{A}) - Y_\perp^{k+1}}{A^{k+1} - A^*} \\
    &\leq f^*(-A^*) - f^*(-A^{k+1}) + \frac{\alpha+\beta}{2} \|Y_\perp^{k+1}\|^2 \\ 
    & \quad + \Big(\frac{1}{2\alpha} - \frac{1}{2L}\Big) + \|\tilde{A}^{k+1} + \rho \Delta Y^{k+1}\|^2 \\
    &\quad + \Big(\frac{1}{2\beta} + \frac{1}{2\mu}\Big)\|\rho \Delta Y^{k+1}\|^2.
\end{align*}
Substituting this in \eqref{eq:bound_(A^k+1-A^k)^T(A^k+1-A^*)}, and then substituting \eqref{eq:bound_(A^k+1-A^k)^T(A^k+1-A^*)} in \eqref{eq:bound_a_tilde^k+1}, we obtain
\begin{align}
    \label{eq:bound_a_tilde^k+1_second}
    &\|\tilde{A}^{k+1}\|^2 \nonumber \\
    &\leq \|\tilde{A}^k\|^2 + 2\rho(f^*(-A^*) - f^*(-A^{k+1})) + \rho(\alpha+\beta)\|Y_\perp^{k+1}\|^2 \nonumber \\
    &+ \rho\Big(\frac{1}{\alpha} - \frac{1}{L}\Big)\|\tilde{A}^{k+1} + \rho \Delta Y^{k+1}\|^2 + \rho\Big(\frac{1}{\beta} + \frac{1}{\mu}\Big)\|\rho \Delta Y^{k+1}\|^2.
\end{align}
Choose $\alpha = L$. Further, since $f^*$ is $1/L$-strongly convex,
\begin{align*}
    f^*(-A^{k+1}) &\geq f^*(-A^*) + \ip{\nabla f^*(-A^*)} {-A^{k+1} + A^*} \\
    &\quad + \frac{\|A^{k+1} - A^*\|^2}{2L}. 
\end{align*}
where ${\nabla f^*(-A^*)}{-A^{k+1} + A^*} = 0$, since $\nabla f^*(-A^*) = X^* = (\onen)X^*$ and $(\onen)A^{k+1} = (\onen)A^* = 0$. Thus,
\begin{align*}
    f^*(-A^*) - f^*(-A^{k+1}) &\leq -\frac{1}{2L}\|\tilde{A}^{k+1}\|^2.
\end{align*}
Substituting the inequality above in \eqref{eq:bound_a_tilde^k+1_second}, dividing by $(\lambda^{k+1})^2$ on both sides and taking supremum over $k = 0, \dots, K-1$, we obtain
\begin{align*}
    &\Big(1 + \frac{\rho}{L} - \frac{1}{\lambda^2}\Big)(\|\tilde{A}\|^{\lambda,K})^2 \\
    &\leq \rho(L+\beta)(\|Y_\perp\|^{\lambda,K})^2 + \rho^3\Big(\frac{1}{\beta} + \frac{1}{\mu}\Big)(\|\Delta Y\|^{\lambda,K})^2 + c_3
\end{align*}
for all $K \geq 0$, for some constant $c_3 \geq 0$. From \eqref{eq:rho_condition_1}, we know that $1 + \rho/L - 1/\lambda^2 > 0$. Hence, using Lemma \ref{lem:square_root_sequences_inequality}, we have
\begin{align}
\label{eq:a_tilde_bound}
    \sqrt{1 + \frac{\rho}{L} - \frac{1}{\lambda^2}}\|\tilde{A}\|^{\lambda,K} &\leq \sqrt{\rho(L+\beta)}\|Y_\perp\|^{\lambda,K} \nonumber \\
    &\quad + \sqrt{\rho^3\Big(\frac{1}{\beta} + \frac{1}{\mu}\Big)}\|\Delta Y\|^{\lambda,K} + \sqrt{c_3}
\end{align}
for all $K \geq 0$. 
Substituting \eqref{eq:delta_y_bound} in \eqref{eq:a_tilde_bound}, we obtain that
\begin{align*}
    &\sqrt{1 + \frac{\rho}{L} - \frac{1}{\lambda^2}}\|\tilde{A}\|^{\lambda,K} \\
    &\leq \left(\sqrt{\rho(L+\beta)} + \sqrt{\rho^3\Big(\frac{1}{\beta} + \frac{1}{\mu}\Big) } \left(\frac{\mu(1 + \lambda)}{\lambda \mu - 2 \rho (1 + \lambda)}\right) \right)\times \\
    & \|Y_\perp\|^{\lambda,K} + \sqrt{\rho^3\Big(\frac{1}{\beta} + \frac{1}{\mu}\Big)} \left(\frac{2(1 + \lambda)}{\lambda \mu - 2 \rho (1 + \lambda)}\right)\|\tilde{A}\|^{\lambda,K} + c_4
\end{align*}
for all $K \geq 0$, for some $c_4 \geq 0$. Rearranging the terms, we obtain that
\begin{align*}
    \bar{c} \|\tilde{A}\|^{\lambda,K} &\leq \Bigg(\sqrt{\rho(L+\beta)} + \sqrt{\rho^3\Big(\frac{1}{\beta} + \frac{1}{\mu}\Big) } \times \\
    & \quad \Big(\frac{\mu(1 + \lambda)}{\lambda \mu - 2 \rho (1 + \lambda)}\Big) \Bigg)\|Y_\perp\|^{\lambda,K} + c_4
\end{align*}
where 
\begin{align*}
    \bar{c} := \sqrt{1 + \frac{\rho}{L} - \frac{1}{\lambda^2}} - \sqrt{\rho^3\Big(\frac{1}{\beta} + \frac{1}{\mu}\Big)} \left(\frac{2(1 + \lambda)}{\lambda \mu - 2 \rho (1 + \lambda)}\right).
\end{align*}
It can be verified that given \eqref{eq:rho_condition_1}, since 
\begin{align*}
    \rho < \frac{\lambda^2\mu^2}{4(1 + \lambda)\left(\lambda\mu + L(1+\lambda)\left(\frac{1}{\beta} + \frac{1}{\mu}\right)\right)},
\end{align*}
we have $\bar{c} > 0$. Hence,
\begin{align}
\label{eq:a_tilde_bound_final}
    \|\tilde{A}\|^{\lambda,K} \leq c_5 \|Y_\perp\|^{\lambda,K} + c_6
\end{align}
for all $K \geq 0$, for some $c_6 \geq 0$ and
\begin{align*}
    c_5 &:= \frac{1}{\bar{c}}\left(\sqrt{\rho(L+\beta)} + \sqrt{\rho^3\Big(\frac{1}{\beta} + \frac{1}{\mu}\Big) } \left(\frac{\mu(1 + \lambda)}{\lambda \mu - 2 \rho (1 + \lambda)}\right) \right) \\
    & > 0.
\end{align*}
Substituting the inequality above in \eqref{eq:delta_y_bound}, we obtain
\begin{align*}
    \|\Delta Y\|^{\lambda,K} &\leq \left(\frac{\mu(1+\lambda)}{\lambda \mu - 2 \rho (1 + \lambda)} + c_5 \frac{2}{\mu} \left(\frac{\mu(1 + \lambda)}{\lambda \mu - 2 \rho (1 + \lambda)}\right)\right) \times \\
    & \quad \|Y_\perp\|^{\lambda,K} + c_7
\end{align*}
for all $K \geq 0$, for some $c_7 \geq 0$. Thus, for all $K \geq 0$,
\begin{align*}
    \|\tilde{A} + \rho \Delta Y\|^{\lambda,K} \leq \|\tilde{A}\|^{\lambda,K} + \rho \|\Delta Y\|^{\lambda,K} \leq \gamma_3 \|Y_\perp\|^{\lambda,K} + c_8
\end{align*}
for some $c_8 \geq 0$ and
\begin{align}
\label{eq:gamma_3_appendix}
    \gamma_3 &:= c_5 + \rho\left(\frac{\mu(1+\lambda)}{\lambda \mu - 2 \rho (1 + \lambda)} + c_5 \frac{2}{\mu} \left(\frac{\mu(1 + \lambda)}{\lambda \mu - 2 \rho (1 + \lambda)}\right)\right) \nonumber \\
    &= c_5 + \rho \left(1 + \frac{2c_5}{\mu}\right) \left(\frac{\mu(1+\lambda)}{\lambda \mu - 2 \rho (1 + \lambda)}\right).
\end{align}

\section{Choice of $\lambda$, $\rho$ and $B$ which ensures $\gamma_1 \gamma_2 \gamma_3 < 1$}
\label{sec:appendix_choice_of_B}

Recall that $\gamma_1 := 1/\mu$, $\gamma_2 := \delta^B$ and $\gamma_3$ is as defined in \eqref{eq:gamma_3_appendix}. First, we choose an arbitrary $\beta > 0$ and a desired $\lambda \in (0,1)$ satisfying \eqref{eq:lambda_condition_1}. Then, we choose a $\rho$ satisfying \eqref{eq:rho_condition_1}. Note that $\gamma_2 = \delta^B$ where $\delta < 1$, while $\gamma_1$ and $\gamma_3$ are independent of $B$. Hence, choosing $B \geq \max\{1,\lceil\log(\gamma_1\gamma_3)/\log(1/\delta)\rceil\}$ ensures $\gamma_1 \gamma_2 \gamma_3 < 1$.

\section{Proof of the additional inequalities in \eqref{eq:additional_inequalities}}
\label{sec:appendix_add_ineq}

From \eqref{eq:a_tilde_bound_final}, we have $\|\tilde{A}\|^\lambda \leq c_5 \|Y_\perp\|^\lambda + c_6$, and from \eqref{eq:x_perp_lamba_K_bound}, we have $\|\tilde{X}\|^\lambda \leq (1/\mu) \|\tilde{A} + \rho \Delta Y\|^\lambda + \|\tilde{X}^0\|$.


\end{document}